%% file: qlargeset_8_4_2_N3.tex
\documentclass{scrartcl}
\title{A new series\\of large sets of subspace designs\\over the binary field}
\author{Michael Kiermaier, Reinhard Laue, Alfred Wassermann
\\
Mathematisches Institut, University of Bayreuth \\
D-95440 Bayreuth, Germany\\
\small{\texttt{\{michael.kiermaier, reinhard.laue, alfred.wassermann\}@uni-bayreuth.de}}
}
\usepackage[utf8]{inputenc}
\usepackage{amsmath}
\usepackage{amssymb}
\usepackage{multicol}
\usepackage[amsmath,thmmarks,hyperref,amsthm]{ntheorem}
\usepackage{array}

\newtheorem{theorem}{Theorem}
\newtheorem{lemma}{Lemma}[section]

\newtheorem{proposition}[lemma]{Proposition}

\theoremstyle{definition}

\theoremstyle{remark}
\newtheorem{remark}[lemma]{Remark}

\newcommand{\B}{\mathcal{B}}
\newcommand{\Des}{\mathcal{D}}
\newcommand{\qbinom}[3]{\genfrac{[}{]}{0pt}{}{#1}{#2}_{#3}}

\newcommand{\LS}{\operatorname{LS}}
\newcommand{\GF}{\operatorname{GF}}
\newcommand{\PG}{\operatorname{PG}}
\newcommand{\GL}{\operatorname{GL}}
\newcommand{\Aut}{\operatorname{Aut}}
\begin{document}
\maketitle
%%%%%%%%%%%%%%%%%%%%%%%%%
\begin{abstract}
In this article, we show the existence of large sets $\LS_2[3](2,k,v)$ for infinitely many values of $k$ and $v$.
The exact condition is $v \geq 8$ and $0 \leq k \leq v$ such that for the remainders $\bar{v}$ and $\bar{k}$ of $v$ and $k$ modulo $6$ we have $2 \leq \bar{v} < \bar{k} \leq 5$.

The proof is constructive and consists of two parts.
First, we give a computer construction for an $\LS_2[3](2,4,8)$, which is a partition of the set of all $4$-dimensional subspaces of an $8$-dimensional vector space over the binary field into three disjoint $2$-$(8, 4, 217)_2$ subspace designs.
Together with the already known $\LS_2[3](2,3,8)$, the application of a recursion method based on a decomposition of the Graßmannian into joins yields a construction for the claimed large sets.
\end{abstract}
%%%%%%%%%%%%%%%%%%%%%%%%%

\section{Introduction}
Let $V$ be a vector space of dimension $v$ over a finite field $\GF(q)$.
For simplicity, a subspace of $V$ of dimension $k$ will be called an \emph{$k$-subspace}.
A \emph{(simple) $t$-$(v,k,\lambda)_q$ subspace design} $\Des = (V,\B)$
consists of a set $\B$ of $k$-subspaces of $V$, called blocks, such that each $t$-subspace of $V$ lies in exactly $\lambda$ blocks.
This notion is a vector space analog of combinatorial $t$-designs on finite sets.
For that reason, subspace designs are also called \emph{$q$-analogs of designs}.
Further names found in the literature include \emph{designs over finite fields}, \emph{designs in vector spaces} and \emph{designs in the $q$-Johnson scheme}.

While combinatorial $t$-designs and Steiner systems have been studied since the 1830s and have a rich literature \cite{Colbourn:2006}, the notion of subspace designs has been introduced by Cameron \cite{Cameron1, Cameron2} and Delsarte \cite{Del76} in the 1970s.

The set of all $k$-subspaces of $V$ is always a design, called trivial design.
In 1987, Thomas \cite{Tho87} constructed the first non-trivial subspace design for $t=2$.
Since then, more subspace designs have been constructed, see
\cite{BKL05,BKKL14,BKOW14,Ito98,KL15,MMY95,Suz90,Suz92}. 

A partition of the trivial design into $N$ disjoint $t$-$(v, k, \lambda)_q$ designs is called \emph{large set} and denoted by $\LS_q[N](t, k, v)$.
The value $\lambda$ is omitted in the parameter notation of a large set as $\lambda = \qbinom{v-t}{k-t}{q}$ is already determined by the other parameters.

For ordinary combinatorial $t$-designs, large sets with $t=1$ exist if and only if $k$ divides $v$ \cite{Bar75}.
In the $q$-analog case, this question is wide open, as it includes the question for the existence of parallelisms in projective geometries:
A large set of $1$-$(v,k,1)_q$ designs (which are called \emph{spreads}) is known as a \emph{$(k-1)$-parallelism} of the projective geometry $\PG(v-1, q)$.
To our knowledge, the only known existence results are the following:
If $v \geq 2$ is a power of $2$, then all $1$-parallelisms do exist \cite{Den72,Beu74}.
Furthermore, for $q=2$ and $v$ even, all $1$-parallelisms do exist \cite{Bak76,Wet91}.
In \cite{EV12}, a $1$-parallelism of $\PG(5,3)$ is given.
The only known parallelism with $k > 1$ is a $2$-parallelism of $\PG(5,2)$ \cite{Sar02}.

For $t \geq 2$, only the following large sets of subspace designs are known:
There are computational constructions of an $\LS_3[2](2,3,6)$ \cite{Bra05}, an $\LS_2[3](2,3,8)$ \cite{BKOW14} and an $\LS_5[2](2,3,6)$ \cite{BKKL14}.
In \cite{BKKL14}, a recursive construction method was developed, which can be seen as a $q$-analog of the theory in \cite{Aj96}, surveyed in \cite{KT09}.
The application of this method to the $\LS_3[2](2,3,6)$ and $\LS_5[2](2,3,6)$ gave an infinite series of large sets $\LS_q(2,k,v)$ with $q\in\{3,5\}$, $v \geq 6$, $v \equiv 2\pmod{4}$, and $3\leq k\leq v-3$, $k\equiv 3\pmod{4}$.

In this paper, we give a computational construction of an $\LS_2[3](2,4,8)$.
The application of the recursion machinery of \cite{BKKL14} to this large set and the already known $\LS_2[3](2,3,8)$ yields a new infinite series of large sets of subspace designs:

\begin{theorem}
	\label{thm:series}
	Let $v$ and $k$ be integers with $v \geq 8$ and $0 \leq k \leq v$ such that
    \renewcommand{\theenumi}{(\roman{enumi})}%    
	\begin{enumerate}
		\item $v\equiv 2\bmod 6$ and $k \equiv 3,4,5\bmod 6$ or
		\item $v\equiv 3\bmod 6$ and $k \equiv 4,5\bmod 6$ or
		\item $v\equiv 4\bmod 6$ and $k \equiv 5\bmod 6$.
	\end{enumerate}
	Then there exists an $\LS_2[3](2,k,v)$.
\end{theorem}

Denoting the remainder of an integer $n$ modulo $6$ by $\bar{n}\in\{0,\ldots,5\}$, the conditions of Theorem~\ref{thm:series}~(i),~(ii)~and~(iii) can be stated as $2 \leq \bar{v} < \bar{k} \leq 5$.

\section{Preliminaries}
\label{sect:prel}
The set of all $k$-subspaces of $V$ is called the \emph{Gra{\ss}mannian} and is denoted by $\qbinom{V}{k}{q}$.
Our focus lies on the case $q = 2$, where the $1$-subspaces $\langle\mathbf{x}\rangle_{\GF(2)}\in \qbinom{V}{1}{2}$ are in one-to-one correspondence with the nonzero vectors $\mathbf{x} \in V\setminus\{\mathbf{0}\}$.
The number of all $k$-subspaces of $V$ is given by the Gaussian binomial coefficient
$$
	\#\qbinom{V}{k}{q} = \qbinom{v}{k}{q} = \begin{cases}
	\frac{(q^v-1)\cdots(q^{v-k+1}-1)}{(q^k-1)\cdots(q-1)} & \text{if }k\in\{0,\ldots,v\}\text{;}\\0 & \text{else.}\end{cases}
$$
The set $\mathcal{L}(V)$ of all subspaces of $V$ forms the subspace lattice of $V$.

By the fundamental theorem of projective geometry, for $v\neq 2$ the automorphism group of $\mathcal{L}(V)$ is given by the natural action of $\operatorname{P\Gamma L}(V)$ on $\mathcal{L}(V)$.
The \emph{automorphism group} $\Aut(D)$ of a subspace design $D = (V,\mathcal{B})$ is defined as the stabilizer of $\mathcal{B}$ under this group action.
Furthermore, for any subgroup $G \leq \operatorname{P\Gamma L}(V)$ we say that $D$ is \emph{$G$-invariant} if $D^G = D$ or equivalently, $G \leq \Aut(D)$.
In the case that $q$ is prime, the group $\operatorname{P\Gamma L}(V)$ reduces to $\operatorname{PGL}(V)$, and for the case of our interest $q = 2$, it reduces further to $\GL(V)$.
After a choice of a basis of $V$, its elements are represented by the invertible $v\times v$ matrices $A$, and the action on $\mathcal{L}(V)$ is given by the vector-matrix-multiplication $\mathbf{v} \mapsto \mathbf{v} A$.

As $t$-design, the trivial design $(V, \qbinom{V}{k}{q})$ has parameters $t$-$(v,k,\lambda_{\max})_q$, where 
$$
    \lambda_{\max}=\qbinom{v-t}{k-t}{q}.
$$
Hence, an obvious necessary condition for the existence of an $\LS_q[N](t,k,v)$ is the equality $\lambda\cdot N = \lambda_{\max}$. 
Moreover, since the blocks of a $t$-design also form an $i$-design for $i\in\{0,\ldots,t\}$, 
we have the necessary conditions
$$
    N \mid \qbinom{v-i}{k-i}{q}\quad\mbox{for }  i\in\{0,\ldots,t\}\text{.}
$$
A parameter set $\LS_q[N](t,k,v)$ is called \emph{admissible} if all the necessary conditions are fulfilled.
If moreover an $\LS_q[N](t,k,v)$ actually exists, the parameter set is called \emph{realizable}.
In the following, it proves useful to extend the parameters to the value $t=-1$ by unconditionally accepting all the large set parameters of the form $\LS_q[N](-1,k,v)$ as admissible and realizable.

By~\cite[Cor.~19]{KL15}, the existence of some $\LS_q[N](t,k,v)$ with $t\geq 1$ implies the existence of \emph{derived large sets} $\LS_q[N](t-1,k-1,v-1)$, \emph{residual large sets} $\LS_q[N](t-1,k,v-1)$ and the \emph{dual large set} $\LS_q[N](t,v-k,v)$.
Furthermore, the existence of $\LS_q[N](t,k-1,v-1)$ and $\LS_q[N](t,k,v-1)$ implies the existence of an $\LS_q[N](t,k,v)$ \cite[Cor.~20]{KL15}.

The following theory is needed for the recursive construction of large sets.
For more details and proofs, see \cite{BKKL14}.

Two subsets $\mathcal{B}_1$ and $\mathcal{B}_2$ of $\qbinom{V}{k}{q}$ are called \emph{$t$-equivalent} if
\[
	\#\{B\in\mathcal{B}_1 \mid T \leq B\} = \#\{B\in\mathcal{B}_2 \mid T \leq B\}
\]
for all $t$-subspaces $T$ of $V$.
In this situation, the pair $(\mathcal{B}_1,\mathcal{B}_2)$ has also been called \emph{trade} or \emph{bitrade}, see for example \cite{KMP16,K15} for some recent results.
Furthermore, given integers $0 \leq t \leq k \leq v$, $N \geq 2$ and a set $\mathcal{B}$ of $k$-subspaces of $V$, a partition $\{\mathcal{B}_1,\ldots,\mathcal{B}_N\}$ of $\mathcal{B}$ is called an \emph{$(N,t)$-partition}, if the parts $\mathcal{B}_i$ are pairwise $t$-equivalent.
The notion of $(N,t)$-partitions can be seen as a generalization of large sets, as by \cite[Lemma~4.8]{BKKL14}, $(V,\{\mathcal{B}_1,\ldots,\mathcal{B}_N\})$ is an $\LS_q[N](t,k,v)$ if and only if $\{\mathcal{B}_1,\ldots,\mathcal{B}_N\}$ is an $(N,t)$-partition of $\qbinom{V}{k}{q}$.

A set $\mathcal{B}$ of $k$-subspaces is called \emph{$(N,t)$-partitionable} if there exists an $(N,t)$-partition of $\mathcal{B}$.
Again, the notion is extended to $t=-1$ by unconditionally calling any set of $k$-subspaces of $V$ \emph{$(N,-1)$-partitionable}.
We have the following structure properties:
If $\mathcal{B}$ is $(N,t)$-partitionable, it is $(N,s)$-partitionable for all $s\in\{-1,\ldots,t\}$ \cite[Lemma~4.3]{BKKL14}.
The disjoint union of $(N,t)$-partitionable sets is again $(N,t)$-partitionable \cite[Lemma~4.7]{BKKL14}.
Therefore, we can construct an $\LS_q[N](t,k,v)$ by decomposing $\qbinom{V}{k}{q}$ into $(N,t)$-partitionable sets.

The known constructions of large sets of classical combinatorial designs and subspace designs often rely on a decomposition into so-called joins.
In \cite{BKKL14}, three kinds of joins are provided.
For our purpose, we only need one of them:
For a chain of subspaces $K_1 \leq U_1 \leq U_2 \leq K_2 \leq V$, the \emph{avoiding join} of $K_1$ and $K_2$ with respect to the factor space $F = U_2/U_1$ is given by
\[
	K_1 *_{\bar{F}} K_2/U_2 = \{K \in \mathcal{L}(V)\mid U_1 \cap K = K_1, U_2 + K = K_2, U_1 \cap K = U_2 \cap K\}\text{.}
\]
By \cite[Lemma~3.7]{BKKL14}, $K_1 *_{\bar{F}} K_2/U_2$ consists of $q^{(\dim(U_1) - \dim(K_1))(\dim(K_2) - \dim(U_1))}$ subspaces of $V$ of dimension $k_1 + k_2 - u_1$.
The definition is extended to sets $\mathcal{B}^{(1)} \subseteq \qbinom{U_1}{k_1}{q}$ and $\mathcal{B}^{(2)} \subseteq \qbinom{V/U_2}{\bar{k}_2}{q}$ of subspaces by setting
\[
	\mathcal{B}^{(1)} \ast \mathcal{B}^{(2)}
        = \bigcup_{\substack{B^{(1)}\in\mathcal{B}^{(1)} \\ B^{(2)}\in\mathcal{B}^{(2)}}} B^{(1)} \ast B^{(2)}\text{.}
\]
The power of the join for the construction of $(N,t)$-partitionable sets is rooted in the following lemma.
In~\cite[Lemma~4.10]{BKKL14} it is stated and proven for all three kinds of joins.

\begin{lemma}[Basic Lemma for the avoiding join]
\label{lem:basic}
Let $U_1 \leq U_2 \leq V$ be a chain of subspaces, $k_1 \in\{0,\ldots,\dim(U_1)\}$, $\bar{k}_2\in\{0,\ldots,\dim(V/U_2)\}$ and $N$ a positive integer.
If $\mathcal{B}^{(1)} \subseteq \qbinom{U_1}{k_1}{q}$ is $(N,t_1)$-partitionable and $\mathcal{B}^{(2)}\subseteq \qbinom{V/U_2}{\bar{k}_2}{q}$ is $(N,t_2)$-partitionable with integers $t_1, t_2\geq -1$, then the avoiding join $\mathcal{B}^{(1)}\ast_{\overline{U_2/U_1}}\mathcal{B}^{(2)}$ is $(N,t_1+t_2+1)$-partitionable.
\end{lemma}

By \cite[Theorem~3.19]{BKKL14}, for a maximal chain
\[
        \{\mathbf 0\} = U_0 < U_1<\ldots < U_v = V
\]
of subspaces and $s\in\{0,\ldots,v-k-1\}$, a partition of $\qbinom{V}{k}{q}$ into avoiding joins is given by the disjoint union
\[
        \qbinom{V}{k}{q}
	= \bigcup_{i=0}^k \qbinom{U_{s+i}}{i}{q} *_{\overline{U_{s+i+1}/U_{s+i}}} \qbinom{V/U_{s+i+1}}{k-i}{q}\text{.}
\]
For the construction of large sets, only the dimensions of the involved Graßmannians are relevant.
Reducing the notation to this information, we say that the above decomposition has the \emph{decomposition type}
\[
	\qbinom{v}{k}{}
	= \bigcup_{i=0}^k 
	\qbinom{s+i}{i}{} * \qbinom{v-s-i-1}{k-i}{}\text{.}
\]

By the above discussion, the application of the Basic Lemma~\ref{lem:basic} to this decomposition yields
\begin{proposition}
	\label{prop:avoiding_construction}
	Let $0 \leq t\leq k \leq v$, $N \geq 2$ and $s\in\{0,\ldots,v-k-1\}$ be integers.
	If for each $i\in\{0,\ldots,k\}$ there are integers $t_1,t_2\geq -1$ with $t_1 + t_2 + 1 \geq t$ such that $\LS_q[N](t_1,i,s+i)$ and $\LS_q[N](t_2,k-i,v-s-i-1)$ both are realizable, then $\LS_q[N](t,k,v)$ is realizable.
\end{proposition}

We remark that the proof of the Basic Lemma in~\cite{BKKL14} is constructive, implying that the statement of Proposition~\ref{prop:avoiding_construction} is constructive, too.

\section{The method of Kramer and Mesner}
Fixing a subgroup $G$ of $\operatorname{P\Gamma L}(V)$, the following idea can be used for the construction of $G$-invariant $t$-$(v,k,\lambda)_q$ designs \cite{KM76,MMY95,BKL05}:
The action of $G$ induces partitions $\qbinom{V}{t}{q} = \bigcup_{i=1}^\tau \mathcal{T}_i$ and $\qbinom{V}{k}{q} = \bigcup_{j=1}^\kappa \mathcal{K}_j$ into orbits.
Any $G$-invariant subspace design will have the form $(V,\mathcal{B})$ with $\mathcal{B} = \bigcup_{j\in J} \mathcal{K}_j$ and $J \subseteq \{1,\ldots,\kappa\}$.
In this way, an index set $J\subseteq \{1,\ldots,\kappa\}$ gives a $t$-$(v,k,\lambda)_q$ design if and only if its characteristic vector $\chi_J \in\{0,1\}^\kappa$ is a solution of the system of linear integer equations
\[
	A^G_{t,k} \chi_J = \lambda\mathbf{1}\text{,}
\]
where $\mathbf{1}$ is the all-one vector and $A^G_{t,k} = (a_{ij})$ is the ($\kappa\times\tau$)-matrix with the entries
\[
    a_{ij}
    = \#\{K\in\mathcal{K}_j \mid T_i \leq K\}\text{.}
\]
Here, $T_i\in\mathcal{T}_i$ denotes a set of orbit representatives.
The matrix $A^G_{t,k}$ is called \emph{$G$-incidence matrix}.

For the construction of large sets, we iterate this method as in \cite{BKOW14}:
After finding a solution $\chi_J$, we remove all the columns from $A^G_{t,k}$ having an index $j\in J$.
Now the solutions of the reduced system give precisely the $G$-invariant subspace designs which are disjoint to $D$.
If we succeed in repeating the process until all $G$-orbits $\mathcal{K}_j$ are covered, we have constructed a large set consisting of $t$-$(v,k,\lambda)_q$ designs.
The same approach has been used by Chee et al~\cite{CheeColbournFurinoea:90} to construct large sets of designs over sets.

\section{An $\LS_2[3](2, 4, 8)$}
A $\LS_2[3](2,4,8)$ consists of three mutually disjoint $2$-$(8,4,217)_2$ designs.
It is worth noting that up to now, even the existence of a single design with these parameters was open.

We succeeded in the construction of an $\LS_2[3](2,4,8)$ by the Kramer-Mesner method described above.
For that purpose, we take the $8$-dimensional $\GF(2)$ vector space $V = \GF(2^8)$ and prescribe the group $G = \langle\sigma^5, \phi^2\rangle$ of order $204$, where $\sigma : x \mapsto x \alpha$ with a primitive element $\alpha$ of $\GF(2^8)$ is a Singer cycle and $\phi : x\mapsto x^2$ is the Frobenius automorphism.
After the choice of a suitable basis of $V$, the generators can be written as
$$
\sigma^5 = 
\begin{pmatrix}
0&0&0&0&0&1&0&0 \\
0&0&0&0&0&0&1&0 \\
0&0&0&0&0&0&0&1 \\
1&0&1&1&1&0&0&0 \\
0&1&0&1&1&1&0&0 \\
0&0&1&0&1&1&1&0 \\
0&0&0&1&0&1&1&1 \\
1&0&1&1&0&0&1&1
\end{pmatrix}
\quad \mbox{ and } \quad
\phi^2 = 
\begin{pmatrix}
1&0&0&0&0&0&0&0 \\
0&0&0&0&1&0&0&0 \\
1&0&1&1&1&0&0&0 \\
1&0&1&1&0&0&1&1 \\
0&0&1&1&0&0&1&0 \\
0&0&1&0&1&1&0&1 \\
1&1&1&1&0&0&0&1 \\
0&0&0&1&1&0&0&0
\end{pmatrix}.
$$

In Tables \ref{tab:B1}, \ref{tab:B2}, and \ref{tab:B3} we list the orbit representatives for the three mutually disjoint designs $\Des_1$, $\Des_2$, and $\Des_3$ we found, using the same encoding as in \cite{BKOW14}.
For each representative, the four row vectors 
$$
\begin{bmatrix}
w_0 & w_1 & \ldots& w_7\\
x_0 & x_1 & \ldots& x_7\\
y_0 & y_1 & \ldots& y_7\\
z_0 & z_1 & \ldots& z_7
\end{bmatrix}
$$
spanning a $3$-subspace of $V$, are encoded as a quadruple of positive integers 
$$
[W, X, Y, Z] = \biggl[\sum_{i=0}^7 w_i2^i, \sum_{i=0}^7 x_i2^i, \sum_{i=0}^7 y_i2^i, \sum_{i=0}^7 z_i2^i\biggr].
$$

\begin{table}[!htbp]
\tiny
{\centering \caption{$2$-$(8, 4, 217; 2)$ design $\B_1$}\label{tab:B1}}
\begin{multicols}{6}
\input{d1}
\end{multicols}
\end{table}

\begin{table}[!htbp]
\tiny
{\centering \caption{$2$-$(8, 4, 217; 2)$ design $\B_2$}\label{tab:B2}}
\begin{multicols}{6}
\input{d2}
\end{multicols}
\end{table}

\begin{table}[!htbp]
\tiny
{\centering \caption{$2$-$(8, 4, 217; 2)$ design $\B_3$}\label{tab:B3}}
\begin{multicols}{6}
\input{d3}
\end{multicols}
\end{table}

\section{An infinite series of large sets}
Now we are ready to proof our main result.
\begin{proof}[Proof of Theorem~\ref{thm:series}]
	We proceed by induction on $v$.
	For $v = 8$, an $\LS_2[3](2,4,8)$ was constructed above.
	In~\cite{BKOW14}, an $\LS_2[3](2,3,8)$ was constructed, and its dual large set is an $\LS_2[3](2,5,8)$.
	The repeated application of~\cite[Cor.~20]{KL15} yields the existence of $\LS_2[3](2,4,9)$, $\LS_2[3](2,5,9)$ and $\LS_2[3](2,5,10)$.

	It remains to consider $v \geq 14$.
	By duality, we may assume $k \leq \frac{v}{2}$.
	By $v - k - 1 \geq \frac{v}{2} - 1 \geq 6$ and the realizability statements of Table~\ref{tbl:decomp}, the existence of an $\LS_2[3](2,k,v)$ follows from the application of Proposition~\ref{prop:avoiding_construction} with $s=5$.

	It remains to show the correctness of Table~\ref{tbl:decomp}. 
	For all $i\equiv 3,4,5\bmod 6$, there exists an $\LS_2[3](2,i,5+i)$ by the induction hypothesis.
	By taking a derived large set, we see that for $i\equiv 2\bmod 6$, there exists an $\LS_2[3](1,i,5+i)$ and for $i\equiv 1\bmod 6$, there exists an $\LS_2[3](0,i,5+i)$.
	Similarly, for $i\equiv 0\bmod 6$, we have an $\LS_2[3](2,k-i,v-6-i)$ by the induction hypothesis.
	By derived large sets, for $i\equiv 5\bmod 6$ there exists an $\LS_2[3](1,k-i,v-6-i)$, and for $i\equiv 4\bmod 6$ there exists an $\LS_2[3](0,k-i,v-6-i)$.
\end{proof}
	
	\begin{table}
	\caption{Realizable large sets used in the proof of Theorem~\ref{thm:series}}
	\label{tbl:decomp}
	\noindent\centering
		$\begin{array}{cccc}
			i                & \LS_2[3](t_1,i,5-i) & \LS_2[3](t_2,k-i,v-6-i) & t_1 + t_2 + 1 \\
			\hline
			i\equiv 0\bmod 6 & t_1 = -1 & t_2 = 2 & 2 \\
			i\equiv 1\bmod 6 & t_1 = 0 & t_2 = 1 & 2 \\
			i\equiv 2\bmod 6 & t_1 = 1 & t_2 = 0 & 2 \\
			i\equiv 3,4,5\bmod 6 & t_1 = 2 & t_2 = -1 & 2
		\end{array}$
	\end{table}

\begin{remark}
	We would like to mention that the proof of Theorem~\ref{thm:series} works for any value of $q$ and $N$, provided that there exist $\LS_q[N](2,3,8)$ and $\LS_q[N](2,4,8)$.
	For example, the parameters $\LS_3[7](2,3,8)$ and $\LS_3[7](2,4,8)$ are admissible, but the realizability is open.
	If both large sets actually do exist, then we get an infinite series like in Theorem~\ref{thm:series}.
\end{remark}

Our knowledge for the existence of $\LS_2[3](2,k,v)$ is shown in Table~\ref{tbl:ueberblick}. 
A minus sign indicates that the parameters are not admissible, and a question mark that the parameters are admissible, but the realizability is open.
All known realizability results are covered by Theorem~\ref{thm:series}.
In these cases we display the parameter $k$ in the table.
Because of duality, only the parameter range $3\leq k \leq v/2$ is shown.
Apart from the already known $\LS_2[3](2,3,8)$ (and its dual $\LS_2[3](2,5,8)$), all these realizability results are new.
The smallest open case is given by the admissible parameter set $\LS_2[3](2,6,20)$, which is out of reach of our current construction methods.

\begin{table}
\caption{Admissibility and realizability of $\LS_2[3](2,k,v)$}
\label{tbl:ueberblick}
\centering\begin{small} 
\begingroup
\setlength\tabcolsep{0pt}
 \begin{tabular}{*{35}{>{\centering\arraybackslash}p{3.5mm}}@{\hskip 2mm}|@{\hskip 1mm}c@{\hskip 1mm}|} 
 & & & & & & & & & & & & & & & & & & & & & & & & & & & & & & & & & & & $\mathbf{v}$ \\
 & & & & & & & & & & & & & & & & & & & & & & & & & & & & & & & & & &-& 6 \\
 & & & & & & & & & & & & & & & & & & & & & & & & & & & & & & & & &-& & 7 \\
 & & & & & & & & & & & & & & & & & & & & & & & & & & & & & & & &3& &4& 8 \\
 & & & & & & & & & & & & & & & & & & & & & & & & & & & & & & &-& &4& & 9 \\
 & & & & & & & & & & & & & & & & & & & & & & & & & & & & & &-& &-& &5& 10 \\
 & & & & & & & & & & & & & & & & & & & & & & & & & & & & &-& &-& &-& & 11 \\
 & & & & & & & & & & & & & & & & & & & & & & & & & & & &-& &-& &-& &-& 12 \\
 & & & & & & & & & & & & & & & & & & & & & & & & & & &-& &-& &-& &-& & 13 \\
 & & & & & & & & & & & & & & & & & & & & & & & & & &3& &4& &5& &-& &-& 14 \\
 & & & & & & & & & & & & & & & & & & & & & & & & &-& &4& &5& &-& &-& & 15 \\
 & & & & & & & & & & & & & & & & & & & & & & & &-& &-& &5& &-& &-& &-& 16 \\
 & & & & & & & & & & & & & & & & & & & & & & &-& &-& &-& &-& &-& &-& & 17 \\
 & & & & & & & & & & & & & & & & & & & & & &-& &-& &-& &-& &-& &-& &-& 18 \\
 & & & & & & & & & & & & & & & & & & & & &-& &-& &-& &-& &-& &-& &-& & 19 \\
 & & & & & & & & & & & & & & & & & & & &3& &4& &5& &?& &?& &?& &9& &10& 20 \\
 & & & & & & & & & & & & & & & & & & &-& &4& &5& &?& &?& &?& &?& &10& & 21 \\
 & & & & & & & & & & & & & & & & & &-& &-& &5& &?& &?& &?& &?& &?& &11& 22 \\
 & & & & & & & & & & & & & & & & &-& &-& &-& &?& &?& &?& &?& &?& &?& & 23 \\
 & & & & & & & & & & & & & & & &-& &-& &-& &-& &?& &?& &?& &?& &?& &?& 24 \\
 & & & & & & & & & & & & & & &-& &-& &-& &-& &-& &?& &?& &?& &?& &?& & 25 \\
 & & & & & & & & & & & & & &3& &4& &5& &-& &-& &-& &9& &10& &11& &?& &?& 26 \\
 & & & & & & & & & & & & &-& &4& &5& &-& &-& &-& &-& &10& &11& &?& &?& & 27 \\
 & & & & & & & & & & & &-& &-& &5& &-& &-& &-& &-& &-& &11& &?& &?& &?& 28 \\
 & & & & & & & & & & &-& &-& &-& &-& &-& &-& &-& &-& &-& &?& &?& &?& & 29 \\
 & & & & & & & & & &-& &-& &-& &-& &-& &-& &-& &-& &-& &-& &?& &?& &?& 30 \\
 & & & & & & & & &-& &-& &-& &-& &-& &-& &-& &-& &-& &-& &-& &?& &?& & 31 \\
 & & & & & & & &3& &4& &5& &-& &-& &-& &9& &10& &11& &-& &-& &-& &15& &16& 32 \\
 & & & & & & &-& &4& &5& &-& &-& &-& &-& &10& &11& &-& &-& &-& &-& &16& & 33 \\
 & & & & & &-& &-& &5& &-& &-& &-& &-& &-& &11& &-& &-& &-& &-& &-& &17& 34 \\
 & & & & &-& &-& &-& &-& &-& &-& &-& &-& &-& &-& &-& &-& &-& &-& &-& & 35 \\
 & & & &-& &-& &-& &-& &-& &-& &-& &-& &-& &-& &-& &-& &-& &-& &-& &-& 36 \\
 & & &-& &-& &-& &-& &-& &-& &-& &-& &-& &-& &-& &-& &-& &-& &-& &-& & 37 \\
 & &3& &4& &5& &?& &?& &?& &9& &10& &11& &?& &?& &?& &15& &16& &17& &-& &-& 38 \\
 &-& &4& &5& &?& &?& &?& &?& &10& &11& &?& &?& &?& &?& &16& &17& &-& &-& & 39 \\
-& &-& &5& &?& &?& &?& &?& &?& &11& &?& &?& &?& &?& &?& &17& &-& &-& &-& 40 \\
\end{tabular} 
\endgroup
 \end{small} 
 \end{table}

\section*{Acknowledgement}
The authors would like to acknowledge the financial support provided by COST -- \emph{European Cooperation in Science and Technology}.
The authors are members of the Action IC1104 \emph{Random Network Coding and Designs over GF(q)}.

\bibliographystyle{plain}
\bibliography{qdesign}

\end{document}

%% file: d1.tex
$[1,34,40,192]$
$[1,34,84,128]$
$[1,34,104,16]$
$[1,66,244,40]$
$[1,98,172,16]$
$[1,106,60,128]$
$[1,106,76,16]$
$[1,106,236,240]$
$[1,114,116,72]$
$[1,130,88,32]$
$[1,134,72,176]$
$[1,134,104,80]$
$[1,162,4,248]$
$[1,170,44,112]$
$[1,170,76,208]$
$[1,194,20,88]$
$[1,210,68,88]$
$[1,210,148,200]$
$[1,214,136,160]$
$[1,226,104,144]$
$[1,226,148,56]$
$[1,228,168,240]$
$[1,242,84,136]$
$[2,4,8,112]$
$[2,4,88,128]$
$[2,4,168,112]$
$[2,8,16,224]$
$[2,36,104,80]$
$[3,44,176,192]$
$[5,106,16,128]$
$[9,18,20,160]$
$[9,66,44,176]$
$[9,74,212,96]$
$[9,82,204,32]$
$[9,170,140,64]$
$[10,68,80,96]$
$[13,10,144,64]$
$[13,70,208,160]$
$[13,202,16,160]$
$[17,18,20,136]$
$[17,18,164,40]$
$[17,24,64,128]$
$[17,34,4,104]$
$[17,50,20,152]$
$[17,66,4,200]$
$[17,82,200,160]$
$[17,130,148,24]$
$[17,130,228,56]$
$[17,146,20,136]$
$[17,146,132,152]$
$[17,218,20,96]$
$[17,242,228,104]$
$[18,116,104,128]$
$[19,36,104,128]$
$[21,146,136,160]$
$[21,150,24,96]$
$[25,68,32,128]$
$[33,18,100,184]$
$[33,34,132,112]$
$[33,34,168,192]$
$[33,34,196,184]$
$[33,66,28,128]$
$[33,98,36,48]$
$[33,98,36,176]$
$[33,130,36,40]$
$[33,136,16,64]$
$[33,148,168,192]$
$[33,170,108,48]$
$[33,178,52,232]$
$[33,210,100,200]$
$[34,164,72,208]$
$[35,132,144,192]$
$[37,38,40,80]$
$[37,86,72,128]$
$[41,10,36,144]$
$[41,34,28,64]$
$[41,74,108,144]$
$[41,98,28,128]$
$[41,234,12,208]$
$[49,2,132,200]$
$[49,6,40,64]$
$[49,18,4,104]$
$[49,34,4,184]$
$[49,34,100,40]$
$[49,82,4,248]$
$[49,114,36,152]$
$[49,114,116,216]$
$[49,130,36,72]$
$[49,130,116,200]$
$[53,146,24,192]$
$[57,98,108,128]$
$[57,186,28,64]$
$[65,4,16,128]$
$[65,10,108,112]$
$[65,34,104,112]$
$[65,82,132,168]$
$[65,82,212,216]$
$[65,82,228,104]$
$[65,130,8,32]$
$[65,130,24,160]$
$[65,130,180,56]$
$[65,130,212,72]$
$[65,146,212,248]$
$[65,162,100,16]$
$[65,170,108,176]$
$[65,194,36,216]$
$[65,194,52,8]$
$[65,194,196,168]$
$[65,202,92,160]$
$[66,68,8,128]$
$[66,68,72,224]$
$[66,140,16,160]$
$[67,4,8,208]$
$[67,4,56,128]$
$[67,36,136,208]$
$[67,68,136,224]$
$[67,136,16,32]$
$[67,196,40,48]$
$[69,74,16,224]$
$[69,142,80,96]$
$[69,210,136,96]$
$[73,66,204,144]$
$[73,74,84,32]$
$[73,76,208,224]$
$[73,162,44,176]$
$[73,202,196,96]$
$[73,226,76,176]$
$[81,34,108,128]$
$[81,90,28,96]$
$[81,98,244,136]$
$[81,130,164,88]$
$[81,146,200,224]$
$[81,210,180,120]$
$[82,116,8,128]$
$[85,194,216,224]$
$[85,198,8,224]$
$[89,42,124,128]$
$[89,74,12,32]$
$[97,2,100,80]$
$[97,34,4,248]$
$[97,34,12,208]$
$[97,34,20,248]$
$[97,36,232,176]$
$[97,42,196,144]$
$[97,50,68,168]$
$[97,50,228,168]$
$[97,66,44,80]$
$[97,66,52,24]$
$[97,100,200,240]$
$[97,106,204,240]$
$[97,130,132,136]$
$[97,130,168,80]$
$[97,194,84,104]$
$[97,194,148,136]$
$[97,194,200,144]$
$[97,226,132,144]$
$[97,228,136,16]$
$[98,164,72,208]$
$[98,228,8,112]$
$[101,134,232,16]$
$[105,12,48,128]$
$[105,74,196,176]$
$[105,162,196,240]$
$[113,18,196,232]$
$[113,34,116,232]$
$[113,66,212,120]$
$[113,82,180,72]$
$[113,114,116,232]$
$[113,130,4,136]$
$[113,130,180,248]$
$[113,194,84,248]$
$[113,226,212,232]$
$[113,242,244,200]$
$[115,36,72,128]$
$[129,10,4,16]$
$[129,34,68,40]$
$[129,66,180,216]$
$[129,82,116,88]$
$[129,98,36,168]$
$[129,98,40,16]$
$[129,98,44,208]$
$[129,100,168,80]$
$[129,114,52,152]$
$[129,114,52,200]$
$[129,114,100,8]$
$[129,114,116,232]$
$[129,130,8,64]$
$[129,130,52,184]$
$[129,162,132,88]$
$[129,194,84,184]$
$[129,194,204,208]$
$[129,210,164,56]$
$[129,210,216,32]$
$[129,226,40,144]$
$[129,226,228,40]$
$[129,242,244,168]$
$[130,4,40,192]$
$[130,12,16,32]$
$[130,20,8,96]$
$[130,164,200,80]$
$[131,4,136,240]$
$[131,132,72,32]$
$[133,34,72,240]$
$[133,166,40,176]$
$[133,166,200,208]$
$[133,202,16,96]$
$[137,2,20,224]$
$[137,76,16,160]$
$[137,82,68,160]$
$[137,90,204,32]$
$[137,106,44,176]$
$[137,130,196,96]$
$[137,138,108,208]$
$[137,154,76,160]$
$[137,158,32,192]$
$[137,170,100,80]$
$[137,172,176,192]$
$[137,194,236,176]$
$[138,132,176,64]$
$[138,156,160,192]$
$[141,14,48,64]$
$[145,2,20,184]$
$[145,2,92,32]$
$[145,34,60,64]$
$[145,66,132,152]$
$[145,82,164,136]$
$[145,98,36,88]$
$[145,114,132,72]$
$[145,114,164,152]$
$[145,130,52,232]$
$[145,138,20,64]$
$[145,146,92,32]$
$[145,154,196,160]$
$[145,178,68,248]$
$[145,226,164,232]$
$[147,68,72,32]$
$[147,132,88,224]$
$[149,86,152,224]$
$[149,214,200,32]$
$[153,82,68,96]$
$[153,130,4,224]$
$[153,146,132,160]$
$[161,26,52,64]$
$[161,34,84,168]$
$[161,34,108,240]$
$[161,34,196,112]$
$[161,42,68,240]$
$[161,66,68,80]$
$[161,74,36,16]$
$[161,74,76,208]$
$[161,82,4,120]$
$[161,98,228,16]$
$[161,114,228,168]$
$[161,138,228,16]$
$[161,146,164,136]$
$[161,162,20,200]$
$[161,170,176,192]$
$[161,178,44,192]$
$[161,226,196,40]$
$[161,234,172,112]$
$[161,242,52,104]$
$[162,40,48,64]$
$[163,132,200,16]$
$[163,228,168,176]$
$[163,228,232,48]$
$[165,146,152,192]$
$[169,10,12,112]$
$[169,34,100,144]$
$[169,98,164,240]$
$[169,130,236,176]$
$[169,170,44,176]$
$[173,138,176,64]$
$[177,2,84,248]$
$[177,34,244,248]$
$[177,42,164,64]$
$[177,66,52,120]$
$[177,66,84,136]$
$[177,114,116,72]$
$[177,150,56,192]$
$[177,162,20,8]$
$[177,194,148,184]$
$[177,210,100,248]$
$[177,210,116,72]$
$[177,226,52,40]$
$[177,226,84,120]$
$[185,178,44,64]$
$[193,2,8,160]$
$[193,10,228,48]$
$[193,34,204,208]$
$[193,68,232,112]$
$[193,98,4,40]$
$[193,178,4,8]$
$[193,178,20,168]$
$[193,194,20,232]$
$[193,196,136,240]$
$[193,226,72,48]$
$[193,226,84,40]$
$[193,226,108,176]$
$[194,36,232,112]$
$[194,68,136,112]$
$[197,102,168,144]$
$[197,130,80,96]$
$[197,150,8,32]$
$[197,194,200,208]$
$[197,206,144,160]$
$[197,214,152,224]$
$[198,8,16,160]$
$[201,82,28,32]$
$[201,162,100,176]$
$[209,26,212,96]$
$[209,50,164,216]$
$[209,66,84,216]$
$[209,114,116,216]$
$[209,130,84,8]$
$[209,146,52,248]$
$[209,148,8,96]$
$[209,154,68,224]$
$[209,178,36,152]$
$[209,178,68,136]$
$[209,212,8,96]$
$[209,242,164,24]$
$[217,26,68,224]$
$[217,194,68,96]$
$[217,210,204,160]$
$[225,10,12,48]$
$[225,38,136,112]$
$[225,98,164,48]$
$[225,170,132,144]$
$[225,170,204,112]$
$[225,210,100,40]$
$[225,230,136,144]$
$[226,36,232,144]$
$[227,36,136,80]$
$[227,132,168,144]$
$[229,6,72,144]$
$[233,130,4,16]$
$[233,226,12,144]$
$[233,226,68,80]$
$[241,2,180,40]$
$[241,50,84,136]$
$[241,50,244,216]$
$[241,66,132,136]$
$[241,82,148,136]$
$[241,114,164,104]$
$[241,162,4,248]$
$[241,242,180,56]$

%% file: d2.tex
$[1,2,80,32]$
$[1,2,140,160]$
$[1,4,232,240]$
$[1,6,104,48]$
$[1,6,136,144]$
$[1,10,204,112]$
$[1,42,112,128]$
$[1,42,164,240]$
$[1,66,100,176]$
$[1,66,244,8]$
$[1,82,148,200]$
$[1,90,12,128]$
$[1,98,40,128]$
$[1,132,40,48]$
$[1,134,104,240]$
$[1,146,20,224]$
$[1,146,196,168]$
$[1,162,100,48]$
$[1,170,172,240]$
$[1,182,56,64]$
$[1,194,40,176]$
$[1,194,76,240]$
$[1,196,136,176]$
$[1,202,196,224]$
$[1,210,132,24]$
$[1,210,196,24]$
$[1,226,164,232]$
$[1,230,232,80]$
$[2,4,184,64]$
$[2,84,88,160]$
$[3,68,8,224]$
$[3,68,232,144]$
$[3,132,168,176]$
$[3,136,32,64]$
$[3,204,144,160]$
$[5,22,72,160]$
$[5,48,64,128]$
$[5,72,208,224]$
$[9,10,76,32]$
$[9,18,4,32]$
$[9,42,4,16]$
$[9,66,220,96]$
$[9,90,28,32]$
$[9,90,36,128]$
$[9,130,36,16]$
$[9,130,148,64]$
$[9,206,80,32]$
$[9,218,12,96]$
$[10,156,32,64]$
$[11,68,80,160]$
$[13,42,144,64]$
$[13,142,16,64]$
$[17,2,84,24]$
$[17,2,132,72]$
$[17,34,4,168]$
$[17,34,84,56]$
$[17,66,132,160]$
$[17,98,52,200]$
$[17,130,156,224]$
$[17,134,184,192]$
$[17,226,164,104]$
$[17,242,68,248]$
$[18,4,24,96]$
$[18,148,72,96]$
$[19,84,136,224]$
$[21,134,56,64]$
$[23,72,32,128]$
$[25,58,124,128]$
$[25,66,148,32]$
$[25,90,20,128]$
$[25,90,92,160]$
$[25,154,148,96]$
$[25,194,20,224]$
$[25,210,196,224]$
$[29,74,32,128]$
$[33,6,40,128]$
$[33,6,104,176]$
$[33,18,20,72]$
$[33,90,84,128]$
$[33,98,68,208]$
$[33,130,100,72]$
$[33,130,148,88]$
$[33,130,212,88]$
$[33,146,116,136]$
$[33,162,56,192]$
$[33,162,180,24]$
$[33,178,244,200]$
$[33,194,136,80]$
$[33,210,84,72]$
$[33,226,4,56]$
$[35,132,72,112]$
$[35,132,136,112]$
$[35,164,8,208]$
$[37,66,88,128]$
$[41,18,132,192]$
$[41,26,60,128]$
$[41,42,204,144]$
$[41,50,164,192]$
$[41,106,132,112]$
$[41,138,228,144]$
$[42,36,176,64]$
$[49,18,148,248]$
$[49,22,184,64]$
$[49,34,68,88]$
$[49,50,100,232]$
$[49,98,164,136]$
$[49,114,20,216]$
$[49,132,152,64]$
$[49,146,52,8]$
$[49,146,148,8]$
$[49,162,84,24]$
$[49,162,180,88]$
$[51,132,152,64]$
$[65,2,4,88]$
$[65,4,8,144]$
$[65,10,100,112]$
$[65,10,164,112]$
$[65,18,244,248]$
$[65,66,204,16]$
$[65,74,204,208]$
$[65,82,220,32]$
$[65,102,136,80]$
$[65,106,204,16]$
$[65,114,52,216]$
$[65,130,136,240]$
$[65,146,196,224]$
$[65,162,116,168]$
$[65,162,164,176]$
$[65,170,196,176]$
$[65,196,152,160]$
$[65,198,104,80]$
$[66,20,88,32]$
$[66,212,88,96]$
$[67,4,208,96]$
$[67,68,136,96]$
$[67,196,104,80]$
$[67,228,72,208]$
$[69,82,24,32]$
$[69,130,24,32]$
$[73,34,228,48]$
$[73,66,68,16]$
$[73,138,36,48]$
$[73,170,4,208]$
$[73,194,76,160]$
$[73,218,68,224]$
$[73,226,68,144]$
$[81,18,20,56]$
$[81,54,104,128]$
$[81,74,84,128]$
$[81,82,132,216]$
$[81,98,164,200]$
$[81,130,132,40]$
$[81,178,36,72]$
$[81,242,180,216]$
$[83,12,32,128]$
$[85,6,96,128]$
$[85,66,136,32]$
$[85,70,216,32]$
$[85,134,216,224]$
$[89,18,28,160]$
$[89,202,156,32]$
$[97,10,60,128]$
$[97,34,36,184]$
$[97,34,196,72]$
$[97,50,100,232]$
$[97,66,100,232]$
$[97,70,104,16]$
$[97,74,172,48]$
$[97,130,196,184]$
$[97,146,132,72]$
$[97,162,20,24]$
$[97,162,140,208]$
$[97,166,232,112]$
$[97,226,148,24]$
$[97,226,180,88]$
$[97,234,140,112]$
$[99,84,40,128]$
$[99,100,56,128]$
$[101,34,232,16]$
$[101,38,8,144]$
$[101,134,200,240]$
$[105,10,28,128]$
$[105,106,100,48]$
$[105,130,140,240]$
$[105,226,172,208]$
$[105,234,68,112]$
$[113,50,84,168]$
$[113,50,108,128]$
$[113,50,196,8]$
$[113,54,104,128]$
$[113,162,228,248]$
$[113,226,100,104]$
$[113,242,100,184]$
$[115,36,8,128]$
$[129,2,52,192]$
$[129,34,72,240]$
$[129,38,200,176]$
$[129,50,100,88]$
$[129,50,148,184]$
$[129,66,4,8]$
$[129,66,12,80]$
$[129,74,156,96]$
$[129,74,204,224]$
$[129,100,136,208]$
$[129,130,72,96]$
$[129,130,212,120]$
$[129,134,136,240]$
$[129,146,20,40]$
$[129,162,40,80]$
$[129,210,88,96]$
$[129,226,44,144]$
$[130,36,200,80]$
$[130,76,16,96]$
$[133,70,200,240]$
$[133,102,72,48]$
$[133,166,40,48]$
$[133,194,8,48]$
$[134,8,80,96]$
$[137,10,76,96]$
$[137,34,156,192]$
$[137,106,4,240]$
$[137,130,156,224]$
$[137,132,144,224]$
$[137,138,44,64]$
$[137,154,204,32]$
$[137,218,148,224]$
$[137,226,36,48]$
$[141,202,144,224]$
$[145,2,212,232]$
$[145,18,4,168]$
$[145,82,244,88]$
$[145,86,88,96]$
$[145,98,68,72]$
$[145,98,116,232]$
$[145,98,212,184]$
$[145,114,212,136]$
$[145,130,196,24]$
$[145,170,12,64]$
$[145,194,164,120]$
$[145,202,84,32]$
$[145,226,52,40]$
$[145,242,196,216]$
$[149,6,160,192]$
$[149,18,32,192]$
$[149,194,88,224]$
$[149,214,200,160]$
$[153,82,212,224]$
$[153,194,140,160]$
$[161,2,100,176]$
$[161,2,164,88]$
$[161,18,100,248]$
$[161,42,16,192]$
$[161,50,132,56]$
$[161,66,40,176]$
$[161,82,36,8]$
$[161,130,136,176]$
$[161,130,228,72]$
$[161,138,76,112]$
$[161,138,164,144]$
$[161,164,232,176]$
$[161,194,36,56]$
$[161,210,148,72]$
$[162,132,8,48]$
$[162,164,136,112]$
$[162,180,184,192]$
$[163,36,48,64]$
$[163,196,104,80]$
$[165,50,40,192]$
$[169,130,12,16]$
$[169,138,156,64]$
$[169,186,148,64]$
$[169,194,164,112]$
$[169,202,108,240]$
$[177,10,148,192]$
$[177,18,116,8]$
$[177,50,116,72]$
$[177,82,132,248]$
$[177,162,52,8]$
$[177,162,148,88]$
$[177,178,228,200]$
$[177,194,4,88]$
$[177,242,148,248]$
$[185,2,172,192]$
$[185,146,164,64]$
$[193,2,68,56]$
$[193,2,104,48]$
$[193,6,208,96]$
$[193,10,220,96]$
$[193,74,212,224]$
$[193,98,132,48]$
$[193,130,104,176]$
$[193,134,72,160]$
$[193,146,212,160]$
$[193,178,100,136]$
$[193,178,116,200]$
$[193,194,148,104]$
$[193,202,100,240]$
$[193,226,12,80]$
$[193,234,100,240]$
$[194,68,16,32]$
$[194,164,104,240]$
$[194,164,136,48]$
$[195,148,216,224]$
$[197,66,80,96]$
$[197,70,136,176]$
$[197,230,232,48]$
$[201,14,208,160]$
$[201,138,132,160]$
$[201,202,108,208]$
$[203,132,144,96]$
$[209,18,196,184]$
$[209,18,228,72]$
$[209,34,132,200]$
$[209,66,148,120]$
$[209,98,36,216]$
$[209,146,20,24]$
$[209,150,216,32]$
$[209,194,4,24]$
$[209,210,68,136]$
$[209,226,100,216]$
$[211,68,136,32]$
$[213,22,136,32]$
$[213,70,136,160]$
$[217,218,20,160]$
$[225,2,204,144]$
$[225,18,100,136]$
$[225,18,180,232]$
$[225,34,36,152]$
$[225,36,72,80]$
$[225,66,100,56]$
$[225,66,104,144]$
$[225,98,196,216]$
$[225,130,8,48]$
$[225,130,36,80]$
$[225,130,116,200]$
$[225,166,72,48]$
$[225,194,196,120]$
$[225,198,232,48]$
$[225,202,132,112]$
$[225,226,72,176]$
$[225,242,212,216]$
$[226,228,104,240]$
$[227,100,40,176]$
$[229,102,40,112]$
$[229,102,200,80]$
$[229,130,168,240]$
$[229,130,200,208]$
$[233,66,172,80]$
$[233,130,100,144]$
$[233,162,12,176]$
$[241,18,132,88]$
$[241,66,4,56]$
$[241,82,68,8]$
$[241,146,148,232]$
$[241,194,84,24]$
$[241,194,180,248]$
$[241,226,20,216]$

%% file: d3.tex
$[1,2,4,96]$
$[1,2,76,176]$
$[1,2,96,128]$
$[1,10,32,64]$
$[1,34,20,136]$
$[1,34,100,240]$
$[1,66,164,184]$
$[1,98,212,136]$
$[1,106,108,144]$
$[1,114,196,248]$
$[1,130,108,48]$
$[1,136,160,192]$
$[1,138,68,144]$
$[1,138,140,240]$
$[1,146,180,72]$
$[1,162,12,192]$
$[1,162,132,56]$
$[1,194,20,24]$
$[1,228,136,16]$
$[1,234,236,16]$
$[2,8,32,64]$
$[2,68,40,128]$
$[2,100,136,240]$
$[2,148,72,96]$
$[2,228,72,144]$
$[2,228,136,112]$
$[3,116,88,128]$
$[5,74,48,128]$
$[5,162,56,64]$
$[5,162,104,176]$
$[5,166,200,240]$
$[9,2,196,208]$
$[9,6,96,128]$
$[9,10,12,240]$
$[9,10,44,48]$
$[9,10,68,96]$
$[9,18,156,96]$
$[9,44,16,128]$
$[9,76,208,96]$
$[9,106,4,128]$
$[9,202,76,176]$
$[9,202,80,96]$
$[9,202,148,224]$
$[11,4,16,64]$
$[11,12,48,64]$
$[14,144,160,192]$
$[17,18,92,160]$
$[17,26,44,192]$
$[17,42,60,64]$
$[17,66,148,232]$
$[17,114,84,128]$
$[17,114,164,168]$
$[17,132,152,160]$
$[17,134,72,160]$
$[17,148,200,224]$
$[17,166,136,64]$
$[17,178,100,168]$
$[17,178,188,64]$
$[17,196,152,224]$
$[17,218,12,224]$
$[18,8,96,128]$
$[21,6,136,32]$
$[25,66,148,96]$
$[25,74,20,32]$
$[33,2,72,80]$
$[33,2,76,112]$
$[33,2,176,64]$
$[33,2,212,24]$
$[33,10,76,128]$
$[33,10,76,176]$
$[33,10,100,128]$
$[33,18,4,24]$
$[33,100,104,144]$
$[33,132,8,112]$
$[33,162,4,216]$
$[33,194,12,144]$
$[33,210,52,232]$
$[33,210,84,200]$
$[33,226,236,112]$
$[33,230,72,112]$
$[33,230,200,112]$
$[34,8,144,192]$
$[34,36,168,48]$
$[34,44,16,64]$
$[34,132,200,144]$
$[35,36,184,192]$
$[35,100,168,48]$
$[37,34,40,144]$
$[37,70,136,80]$
$[37,98,40,128]$
$[37,198,72,176]$
$[41,42,132,64]$
$[41,50,188,192]$
$[41,66,4,144]$
$[41,98,84,128]$
$[41,106,108,16]$
$[41,106,172,240]$
$[41,194,196,208]$
$[49,18,180,8]$
$[49,26,4,128]$
$[49,82,116,232]$
$[49,114,180,88]$
$[49,178,212,88]$
$[49,226,164,168]$
$[49,226,164,232]$
$[49,242,68,120]$
$[50,52,72,128]$
$[51,4,24,128]$
$[57,58,28,64]$
$[61,46,64,128]$
$[65,6,104,80]$
$[65,10,236,144]$
$[65,18,156,32]$
$[65,66,8,48]$
$[65,66,180,8]$
$[65,68,136,96]$
$[65,74,172,80]$
$[65,82,100,8]$
$[65,98,132,216]$
$[65,98,148,136]$
$[65,98,164,232]$
$[65,98,228,120]$
$[65,130,228,24]$
$[65,130,244,72]$
$[65,134,144,224]$
$[65,146,68,40]$
$[65,146,196,56]$
$[65,154,212,32]$
$[65,166,104,80]$
$[65,178,116,216]$
$[65,202,12,176]$
$[65,226,148,200]$
$[65,234,68,112]$
$[67,36,72,128]$
$[67,68,72,208]$
$[67,76,16,32]$
$[67,132,8,208]$
$[67,228,136,16]$
$[69,2,88,128]$
$[69,14,208,224]$
$[69,38,232,48]$
$[69,82,88,32]$
$[69,102,232,48]$
$[69,146,152,96]$
$[69,162,136,80]$
$[69,214,72,96]$
$[73,2,4,80]$
$[73,2,16,224]$
$[73,10,12,160]$
$[73,42,172,16]$
$[73,44,112,128]$
$[73,74,16,128]$
$[73,74,148,224]$
$[73,106,44,208]$
$[73,122,100,128]$
$[73,202,220,160]$
$[73,234,236,48]$
$[81,38,56,128]$
$[81,82,148,136]$
$[81,90,140,96]$
$[81,98,52,200]$
$[81,210,116,248]$
$[81,210,244,8]$
$[81,226,244,104]$
$[82,72,32,128]$
$[82,116,40,128]$
$[85,150,88,32]$
$[89,10,204,32]$
$[89,82,84,32]$
$[89,106,52,128]$
$[89,202,4,32]$
$[89,202,220,224]$
$[91,12,96,128]$
$[97,34,164,168]$
$[97,66,12,144]$
$[97,66,80,128]$
$[97,66,100,184]$
$[97,66,104,80]$
$[97,74,228,112]$
$[97,106,140,80]$
$[97,162,100,80]$
$[97,194,236,16]$
$[97,226,100,240]$
$[97,226,136,80]$
$[97,228,40,240]$
$[97,228,200,48]$
$[98,4,104,112]$
$[98,4,136,144]$
$[98,4,168,144]$
$[98,164,104,176]$
$[98,196,8,80]$
$[99,132,232,240]$
$[101,106,16,128]$
$[105,42,204,176]$
$[105,170,4,176]$
$[105,194,76,176]$
$[113,38,40,128]$
$[113,50,196,136]$
$[113,82,84,72]$
$[113,178,132,120]$
$[113,178,228,120]$
$[113,210,116,152]$
$[129,2,4,248]$
$[129,6,216,32]$
$[129,10,12,192]$
$[129,18,132,104]$
$[129,38,16,192]$
$[129,38,176,192]$
$[129,66,68,200]$
$[129,82,20,40]$
$[129,82,132,160]$
$[129,98,68,168]$
$[129,130,92,160]$
$[129,130,116,40]$
$[129,146,20,104]$
$[129,146,52,152]$
$[129,146,84,72]$
$[129,162,196,176]$
$[129,164,48,64]$
$[129,178,116,152]$
$[129,194,232,80]$
$[129,212,88,96]$
$[129,218,4,32]$
$[129,218,92,32]$
$[129,242,4,72]$
$[130,100,136,48]$
$[130,132,200,176]$
$[131,36,200,80]$
$[133,6,176,64]$
$[133,50,56,192]$
$[133,134,48,64]$
$[133,134,232,144]$
$[133,194,200,144]$
$[133,206,208,96]$
$[137,2,204,96]$
$[137,2,236,176]$
$[137,18,156,96]$
$[137,34,36,16]$
$[137,42,4,208]$
$[137,162,236,144]$
$[137,194,68,48]$
$[137,194,140,224]$
$[137,202,4,224]$
$[137,226,164,240]$
$[141,202,208,96]$
$[145,26,84,96]$
$[145,50,196,184]$
$[145,52,168,192]$
$[145,114,36,72]$
$[145,146,76,96]$
$[145,146,180,64]$
$[145,162,116,8]$
$[145,178,36,40]$
$[145,178,52,192]$
$[145,226,244,168]$
$[145,242,212,168]$
$[153,42,60,192]$
$[161,2,136,240]$
$[161,34,40,80]$
$[161,34,136,176]$
$[161,36,200,16]$
$[161,54,8,64]$
$[161,66,100,216]$
$[161,74,108,240]$
$[161,98,148,248]$
$[161,100,72,16]$
$[161,178,148,8]$
$[161,194,180,232]$
$[162,36,72,80]$
$[162,44,48,192]$
$[163,100,104,176]$
$[163,132,104,112]$
$[169,34,140,208]$
$[169,34,164,208]$
$[169,106,228,48]$
$[169,154,180,64]$
$[169,186,52,64]$
$[169,194,140,208]$
$[177,6,56,64]$
$[177,18,132,136]$
$[177,34,52,192]$
$[177,42,132,192]$
$[177,66,132,152]$
$[177,66,228,136]$
$[177,114,132,8]$
$[177,114,180,24]$
$[177,134,24,192]$
$[177,146,196,88]$
$[177,162,52,136]$
$[177,162,140,192]$
$[177,194,132,104]$
$[177,226,212,216]$
$[179,20,8,192]$
$[181,22,184,192]$
$[181,178,184,192]$
$[185,138,172,192]$
$[193,38,72,16]$
$[193,50,20,88]$
$[193,66,140,32]$
$[193,68,8,48]$
$[193,82,200,224]$
$[193,106,76,144]$
$[193,114,196,120]$
$[193,114,228,168]$
$[193,140,208,96]$
$[193,162,168,16]$
$[193,178,212,56]$
$[193,194,232,16]$
$[193,196,136,224]$
$[193,196,152,96]$
$[193,226,236,16]$
$[193,226,244,136]$
$[194,68,104,176]$
$[195,4,80,32]$
$[197,18,24,32]$
$[197,66,136,240]$
$[197,150,24,224]$
$[197,214,152,32]$
$[197,230,40,176]$
$[201,2,76,224]$
$[201,10,4,208]$
$[201,12,144,32]$
$[201,132,144,96]$
$[201,138,208,224]$
$[201,170,140,144]$
$[209,130,152,32]$
$[209,130,164,200]$
$[209,148,72,96]$
$[209,162,244,184]$
$[209,194,8,96]$
$[209,210,212,216]$
$[213,134,24,96]$
$[225,2,136,112]$
$[225,4,232,176]$
$[225,34,4,56]$
$[225,34,12,80]$
$[225,34,12,208]$
$[225,34,148,40]$
$[225,34,164,240]$
$[225,36,232,16]$
$[225,98,180,200]$
$[225,114,36,72]$
$[225,130,164,112]$
$[225,138,172,16]$
$[225,194,4,168]$
$[225,194,84,40]$
$[225,210,100,168]$
$[225,226,100,200]$
$[226,132,72,48]$
$[229,134,104,48]$
$[233,106,68,48]$
$[233,106,140,176]$
$[233,106,164,240]$
$[233,170,44,176]$
$[241,82,52,136]$
$[241,98,20,8]$
$[241,114,116,168]$
$[241,210,20,104]$